\documentclass[10pt]{article}

\usepackage[margin=1in]{geometry}  
\usepackage{graphicx}              
\usepackage{graphicx}              
\usepackage{amsmath}               
\usepackage{amsfonts}              
\usepackage{amsthm}                
\usepackage{amssymb}

\usepackage{comment}

\newtheorem{theorem}{Theorem}[section]
\newtheorem{lemma}[theorem]{Lemma}

\newtheorem{corollary}[theorem]{Corollary}

\newtheorem{remark}[theorem]{Remark}

\begin{document}

\nocite{*}

\title{Rational points on non-normal hypercubics}

\author{Evgeny Mayanskiy}

\maketitle

\begin{abstract}
We show that the count of rational points by de la Bret\`{e}che, Browning and Salberger on the Cayley ruled cubic surface extends to all non-normal integral hypercubics which are not cones. 
\end{abstract}

\section{Introduction}

In \cite{Browning}, a precise asymptotic formula for the number of rational points on the Cayley ruled cubic surface was established. Moreover, the leading term was expressed in terms of Tamagawa constants. The purpose of this note is to show that the asymptotic formula of \cite{Browning} extends to all non-normal integral hypercubics which are not cones. The fibration method, which was one of the methods used in \cite{Browning}, goes through exactly as in \cite{Browning}.\\

In order to count rational points, we use the following height function:
$$
H\colon {\mathbb P}^n(\mathbb Q) \to {\mathbb R}_{>0},\quad (t_0: \ldots : t_n) \mapsto \sqrt{t_0^2 + \ldots + t_n^2},
$$
where $(t_0, \ldots , t_n) \in {\mathbb Z}^{n+1}_{\operatorname{prim}}$ and ${\mathbb Z}^{n+1}_{\operatorname{prim}}=\{  (t_0, \ldots , t_n) \in  {\mathbb Z}^{n+1}\setminus (0, \ldots , 0) \;\mid \; \operatorname{gcd}(t_0, \ldots , t_n)=1  \}$.\\

Given a geometrically integral projective variety $W\subset {\mathbb P}^n_{\mathbb Q}$, we let $W^{norm}\subset W$ denote the locus of geometrically normal points. The counting function for the rational points on $W^{norm}$, as a function of $B\in {\mathbb R}_{>0}$, is 
$$
N(W^{norm}, B)=\# \{ t\in W^{norm}(\mathbb Q) \; \mid \; H(t)\leq B \}.
$$

Our main result is the following. All asymptotic formulas are given with respect to $B\to +\infty$.

\begin{theorem}\label{main}
\begin{enumerate}
\item Let $a\in \mathbb Z \setminus \{ 0 \} $ be square-free and $W\subset {\mathbb P}^3_{\mathbb Q}$ be given by the equation $t_0t_1t_2 + t_3\cdot (t_0^2+a\cdot t_1^2)=0$. Then 
$$
N(W^{norm}, B)=\frac{\pi B^2}{4\zeta(2)}\cdot \left( 4+\sum_{\substack{(\mu,\lambda)\in {\mathbb Z}^2_{\operatorname{prim}} \\ \mu \neq 0}} \frac{\operatorname{gcd}(a,\mu)}{\sqrt{f(\mu,\lambda)}} \right) + O(B^{3/2}(\operatorname{log}B)^2),
$$
where $f(\mu,\lambda) = ({\lambda}^2 + {\mu}^2)({\lambda}^2{\mu}^2+({\mu}^2 + a\cdot {\lambda}^2)^2)$.
\item Let $W\subset {\mathbb P}^4_{\mathbb Q}$ be given by the equation $t_0^2t_2 + t_1^2t_3 + t_0t_1t_4=0$. Then 
$$
N(W^{norm}, B)=\frac{\pi B^3}{3\zeta(3)}\sum_{(\mu,\lambda)\in {\mathbb Z}^2_{\operatorname{prim}}} \frac{1}{\sqrt{f(\mu,\lambda)}} + O(B^{2}\operatorname{log}B),
$$
where $f(\mu,\lambda) = ({\lambda}^2 + {\mu}^2)({\lambda}^2{\mu}^2+{\mu}^4 + {\lambda}^4)$.
\end{enumerate}
\end{theorem}

Our notation and arguments follow closely \cite{Browning}.

\section{Classification of non-normal hypercubics over the rationals}

Non-normal cubic hypersurfaces over algebraically closed fields are classified in \cite{LPS}. The same argument also gives the following classification over $\mathbb Q$.

\begin{theorem}\label{classification}(cf. \cite{LPS}, Theorem~$3.1$)
Let $W\subset {\mathbb P}^n_{\mathbb Q}$ be a geometrically integral and geometrically non-normal hypersurface given by a homogeneous cubic polynomial $F\in \mathbb Q [ t_0,\ldots , t_n ]$. Then either $W$ is a cone or $F$ can be obtained by a linear coordinate change over $\mathbb Q$ from one of the following polynomials:
\begin{itemize}
\item $(t_0^2+a\cdot t_1^2)t_2+t_1^2(b\cdot t_0+c\cdot t_1)$, $a,b,c\in \mathbb Z$, $n=2$,
\item $t_0t_1t_2+t_0^3+a\cdot t_1^3$, $a\in\mathbb Z$, $n=2$,
\item $t_0^2t_2+t_1^2t_3$, $n=3$,
\item $t_0t_1t_2+t_3(t_0^2+a\cdot t_1^2)$, $a\in \mathbb Z\setminus \{ 0,1 \}$ is square-free, $n=3$,
\item $t_0t_1t_2+t_3t_0^2+t_1^3$, $n=3$,
\item $t_0^2t_2+t_0t_1t_3+t_1^2t_4$, $n=4$.
\end{itemize}
\end{theorem}

Note that $t_0t_1t_2+t_3(t_0^2+t_1^2)$ transforms to $4(t_0^2t_3+t_1^2t_2)$ after substitution $t_0\mapsto t_0+t_1$, $t_1\mapsto t_0-t_1$, $t_2\mapsto 2(t_3-t_2)$, $t_3\mapsto t_2+t_3$. Rational points on the cubic surfaces given by equations $t_0^2t_2+t_1^2t_3=0$ and $t_0t_1t_2+t_3t_0^2+t_1^3=0$ were counted in \cite{Browning}. 

\section{Geometry of the non-normal cubic threefold}

In this section we consider the hypercubic $\overline{W}\subset {\mathbb P}_{\overline{\mathbb Q}}^4$ given by the equation $t_0^2t_2+t_0t_1t_3+t_1^2t_4=0$. The normalization $\nu\colon X\to\overline{W}$ is the projection of the Segre cubic threefold scroll $X={\mathbb P}_{\overline{\mathbb Q}}^1 \times {\mathbb P}_{\overline{\mathbb Q}}^2 \subset {\mathbb P}_{\overline{\mathbb Q}}^5$ from a point $P\in {\mathbb P}_{\overline{\mathbb Q}}^5\setminus X$. \cite{LPS}\\

\begin{theorem}
The automorphism group of $\overline{W}$ fits into the short exact sequence of groups
$$
0\rightarrow K \rightarrow \operatorname{Aut}(\overline{W}) \rightarrow \operatorname{PGL}(2)\rightarrow 0,
$$
where $K={\overline{\mathbb Q}}^{*}\oplus \overline{\mathbb Q}\oplus \overline{\mathbb Q}$ with the product $(a,b,c)\cdot (a',b',c')=(aa',a'b+b',a'c+c')$.
\end{theorem}
\begin{proof}
By the Lefschetz theorem, $\operatorname{Pic}(\overline{W})=\mathbb Z \mathcal{O}(1)$. Hence any automorphism of $\overline{W}$ is induced by an automorphism of ${\mathbb P}_{\overline{\mathbb Q}}^4$, i.e. by $A\in \operatorname{PGL}(5)$ which preserves (up to a scalar multiple) $t_0^2t_2+t_0t_1t_3+t_1^2t_4$. Any such $A$ also preserves the non-normal locus $\{  t_0=t_1=0 \}$ and so induces an automorphism of ${\mathbb P}_{\overline{\mathbb Q}}(\overline{\mathbb Q} t_0\oplus \overline{\mathbb Q} t_1)\cong {\mathbb P}_{\overline{\mathbb Q}}^1$. The resulting map $\operatorname{Aut}(\overline{W}) \rightarrow \operatorname{PGL}(2)$ is surjective.\\

Explicitly, suppose $A(t_0)=at_0+t_1$, $A(t_1)=ct_0+dt_1$, $ad-c\neq 0$. Then
\begin{gather*}
A(t_2)=u_4\cdot (d^2\cdot t_2-cd\cdot t_3+c^2\cdot t_4)-(ca_{31}+cda_{41})\cdot t_0-(da_{31}+d^2a_{41})\cdot t_1,\\
A(t_3)=u_4\cdot (-2d\cdot t_2+(ad+c)\cdot t_3-2ac\cdot t_4)+(aa_{31}+(ad-c)a_{41})\cdot t_0+a_{31}\cdot t_1,\\
A(t_4)=u_4\cdot (t_2-a\cdot t_3+a^2\cdot t_4)+aa_{41}\cdot t_0+a_{41}\cdot t_1,
\end{gather*}
where $u_4\neq 0$ and $a_{31},a_{41}$ are arbitrary. If $A(t_0)=t_0$, $A(t_1)=ct_0+dt_1$, $d\neq 0$, then
\begin{gather*}
A(t_2)=w_4\cdot (d^2\cdot t_2-cd\cdot t_3+c^2\cdot t_4)-(ca_{30}+c^2a_{40})\cdot t_0-(da_{30}+cda_{40})\cdot t_1,\\
A(t_3)=w_4\cdot (d\cdot t_3-2c\cdot t_4)+a_{30}\cdot t_0-da_{40}\cdot t_1,\\
A(t_4)=w_4\cdot t_4+a_{40}\cdot t_0,
\end{gather*}
where $w_4\neq 0$ and $a_{30},a_{40}$ are arbitrary. 
\end{proof}

The following Corollary was inspired by the arguments in \cite{Browning}.\\

\begin{corollary}
$\overline{W}$ is not toric.
\end{corollary}
\begin{proof}(cf. \cite{Browning})
The maximal torus in $\operatorname{Aut}(\overline{W})$ has dimension $2 < \operatorname{dim}(\overline{W})$.\end{proof}

\section{Rational points on $t_0t_1t_2 + t_3\cdot (t_0^2+a\cdot t_1^2)=0$}

In this section we prove Theorem~\ref{main}, part~$1$. The argument and notation follow \cite{Browning} closely. Let $W\subset {\mathbb P}^3_{\mathbb Q}$ be given by the equation $t_0t_1t_2 + t_3\cdot (t_0^2+a\cdot t_1^2)=0$, where $a\in \mathbb Z \setminus \{ 0 \} $ is square-free. Let $V=W\setminus \{ t_0=t_1=0 \}$ and
$$
V_y = \{ (t_0:t_1:t_2:t_3)\in {\mathbb P}^3 \; \mid \; \lambda t_0 - \mu t_1 = \mu\lambda t_2+({\mu}^2 +a{\lambda}^2)t_3=0  \} \subset W,
$$
where $y=(\mu : \lambda ) \in {\mathbb P}^1$.\\

Then $V=\coprod\limits_{y\in {\mathbb P}^1}V\cap V_y$, and so
$$
N(V,B)=\sum_{y\in {\mathbb P}^1(\mathbb Q)}N(V\cap V_y,B).
$$

\textit{Claim}. (cf. \cite{Browning}, Lemma~$3.1$) If $\mu \neq 0$, $(\mu, \lambda)\in {\mathbb Z}^2_{\operatorname{prim}}$, $y=(\mu :\lambda)\in {\mathbb P}^1(\mathbb Q)$, then
$$
N(V\cap V_y,B)=\frac{1}{2}\cdot \# \{\; (\tau_0 , \tau_3)\in {\mathbb Z}^2_{\operatorname{prim}} \;\mid \; \tau_0\neq 0 , \; H(\mu \tau_0,\lambda \tau_0, ({\mu_1}^2d+a_1{\lambda}^2)\cdot \tau_3, (\mu_1 \lambda )\cdot \tau_3)\leq B \},
$$
where $d=\operatorname{gcd}(a,\mu)\geq 1$, $\mu=\mu_1d$, $a=a_1d$.
$$
N(V\cap V_{(0:1)},B)=\frac{1}{2}\cdot \# \{\; (\tau_0 , \tau_3)\in {\mathbb Z}^2_{\operatorname{prim}} \;\mid \; \tau_0\neq 0 , \;\sqrt{{\tau_0}^2+{\tau_3}^2}\leq B  \}.
$$
\begin{proof} \cite{Browning}
Let $\mu\lambda \neq 0$ and take $(t_0,t_1,t_2,t_3)\in {\mathbb Z}^4_{\operatorname{prim}}\cap V_y$. Then $t_0=\mu\tau_0$, $t_1=\lambda\tau_0$, $t_3=(\mu_1\lambda )\cdot\tau_3$, $t_2=-({\mu_1}^2d+a_1{\lambda}^2)\cdot \tau_3$. \end{proof}

Note that $N(V\cap V_y,B)=0$ unless ${\lambda}^2+{\mu}^2\leq B^2$. Moreover, if $\mu\neq 0$ and 
$$
{\lambda}^2+{\mu}^2\leq B^2 < {\lambda}^2+{\mu}^2 + {\mu_1}^2{\lambda }^2 + ({\mu_1}^2d+a_1{\lambda}^2)^2,
$$
then $N(V\cap V_y,B)=1$. In particular, if $N(V\cap V_y,B)> 1$, then $\lvert \mu\lambda \rvert \ll B$ and so $\operatorname{min}\{ \lvert \mu \rvert , \lvert \lambda \rvert \} \ll \sqrt{B}$. Also, for $(\tau_0,\tau_3)\in {\mathbb Z}^2_{\operatorname{prim}}$ contributing to $N(V\cap V_y,B)$,
$$
\lvert\tau_0\rvert \ll \frac{B}{\sqrt{{\lambda}^2+{\mu}^2}},\quad \lvert\tau_3\rvert \ll \frac{B}{\operatorname{max} \{  \lvert \lambda\rvert , \lvert\mu\rvert \} }.
$$

Suppose $\mu\neq 0$. If 
$$
N^{*}(V_y,B)=\# \{\; (\tau_0 , \tau_3)\in {\mathbb Z}^2 \;\mid \; \tau_0\neq 0 ,\; H(\mu \tau_0,\lambda \tau_0, ({\mu_1}^2d+a_1{\lambda}^2)\cdot \tau_3, (\mu_1 \lambda )\cdot \tau_3)\leq B \},
$$
then
$$
N(V\cap V_y,B)=\frac{1}{2}\sum_{k\ll \frac{B}{\sqrt{{\lambda}^2+{\mu}^2}}}\mu(k)\cdot N^{*}\left( V_y,\frac{B}{k}\right) .
$$

By Euler's summation formula,
\begin{multline*}
N^{*}(V_y,B) = \sum_{\lvert \tau_3\rvert \leq \frac{B}{\sqrt{(\mu_1\lambda )^2+({\mu_1}^2d+a_1{\lambda}^2)^2}}} 2\cdot \left[ \frac{\sqrt{B^2 - ((\mu_1\lambda )^2+({\mu_1}^2d+a_1{\lambda}^2)^2)\cdot {\tau_3}^2 }}{\sqrt{{\lambda}^2+{\mu}^2}} \right] \\
=2 \cdot \int\limits_{-\frac{B}{\sqrt{(\mu_1\lambda )^2+({\mu_1}^2d+a_1{\lambda}^2)^2}}}^{\frac{B}{\sqrt{(\mu_1\lambda )^2+({\mu_1}^2d+a_1{\lambda}^2)^2}}}\frac{1}{\sqrt{{\lambda}^2+{\mu}^2}}\cdot \sqrt{B^2 - ((\mu_1\lambda )^2+({\mu_1}^2d+a_1{\lambda}^2)^2)\cdot x^2 }\cdot \operatorname{d}x + O\left( \frac{B}{\operatorname{max} \{ \lvert \mu \rvert , \lvert \lambda \rvert \} }\right) \\
=\frac{\pi B^2}{\sqrt{{\lambda}^2+{\mu}^2}\cdot \sqrt{(\mu_1\lambda )^2+({\mu_1}^2d+a_1{\lambda}^2)^2} } + O\left( \frac{B}{\operatorname{max} \{  \lvert \lambda\rvert , \lvert\mu\rvert \} }\right) .
\end{multline*}

Hence 
$$
N(V\cap V_y,B)=\frac{\pi B^2}{2\zeta(2)}\cdot \frac{1}{\sqrt{({\lambda}^2+{\mu}^2)((\mu_1\lambda )^2+({\mu_1}^2d+a_1{\lambda}^2)^2)} } + O\left( \frac{B\cdot \operatorname{log}B}{\operatorname{max} \{  \lvert \lambda\rvert , \lvert\mu\rvert \} }\right) .
$$

A similar calculation gives
$$
N(V\cap V_{(0:1)},B)=\frac{\pi B^2}{2\zeta(2)} + O(B\cdot \operatorname{log}B).
$$

In the expression 
\begin{multline*}
N(V,B)=N(V\cap V_{(0:1)},B) \\
+ \frac{1}{2}\sum_{d\mid a}\left( \sum_{\substack{(\mu ,\lambda )\in {\mathbb Z}^2_{\operatorname{prim}} \\ \mu \neq 0,\; d=\operatorname{gcd}(a,\mu) \\ {\mu}^2+{\lambda}^2 \leq B^2 < {\mu}^2+{\lambda}^2 + (\mu_1\lambda )^2 +({\mu_1}^2d+a_1{\lambda}^2)^2 }}1 + \sum_{\substack{(\mu ,\lambda )\in {\mathbb Z}^2_{\operatorname{prim}} \\ \mu \neq 0,\; d=\operatorname{gcd}(a,\mu) \\ {\mu}^2+{\lambda}^2 + (\mu_1\lambda )^2 +({\mu_1}^2d+a_1{\lambda}^2)^2 \leq B^2 }} N(V\cap V_y,B) \right) ,
\end{multline*}
denote the first and the second sums over $(\mu , \lambda)$ by $\Sigma_1$ and $\Sigma_2$ respectively. Then 
$$
\Sigma_1 = \sum_{\substack{(\mu ,\lambda )\in {\mathbb Z}^2_{\operatorname{prim}} \\ \mu \neq 0,\; d=\operatorname{gcd}(a,\mu) \\ {\mu}^2+{\lambda}^2 \leq B^2 }}1 + O(B^{3/2}),
$$
and
$$
\Sigma_2 = \frac{\pi B^2}{2\zeta(2)}\sum_{\substack{(\mu ,\lambda )\in {\mathbb Z}^2_{\operatorname{prim}} \\ \mu \neq 0,\; d=\operatorname{gcd}(a,\mu) }} \frac{1}{\sqrt{({\mu}^2+{\lambda}^2)((\mu_1\lambda )^2 +({\mu_1}^2d+a_1{\lambda}^2)^2)}} + O(B^{3/2}\cdot (\operatorname{log}B)^2).
$$

The same calculation as above gives
$$
\sum_{\substack{(\mu ,\lambda )\in {\mathbb Z}^2_{\operatorname{prim}} \\ \mu \neq 0 \\ {\mu}^2+{\lambda}^2 \leq B^2 }}1 = \frac{\pi B^2}{\zeta(2)} +  O(B\cdot \operatorname{log} B).
$$

All together this gives the asymptotic formula in Theorem~\ref{main}, part~$1$.  

\section{Rational points on $t_0^2t_2+t_0t_1t_3+t_1^2t_4=0$}

In this section we prove Theorem~\ref{main}, part~$2$. The argument and notation follow \cite{Browning} closely. Let $W\subset {\mathbb P}^4_{\mathbb Q}$ be given by the equation $t_0^2t_2+t_1^2t_3+t_0t_1t_4=0$. Let $V=W\setminus \{ t_0=t_1=0 \}$ and
$$
V_y = \{ (t_0:t_1:t_2:t_3:t_4)\in {\mathbb P}^4 \; \mid \; \lambda t_0 - \mu t_1 = {\mu}^2 t_2+{\lambda}^2 t_3 +\mu\lambda t_4=0  \} \subset W,
$$
where $y=(\mu : \lambda ) \in {\mathbb P}^1$.\\

Then $V=\coprod\limits_{y\in {\mathbb P}^1}V\cap V_y$, and so
$$
N(V,B)=\sum_{y\in {\mathbb P}^1(\mathbb Q)}N(V\cap V_y,B).
$$

\textit{Claim}. (cf. \cite{Browning}, Lemma~$3.1$) Assume $(\mu, \lambda)\in {\mathbb Z}^2_{\operatorname{prim}}$, $y=(\mu :\lambda)\in {\mathbb P}^1(\mathbb Q)$. Then
$$
N(V\cap V_y,B)=\frac{1}{2}\cdot \# \{\; (\tau_0 , \tau_2 , \tau_3)\in {\mathbb Z}^3_{\operatorname{prim}} \;\mid \; \tau_0\neq 0 , \; H(\mu \tau_0,\lambda \tau_0, \lambda \tau_2, \mu \tau_3 , \lambda \tau_3 +\mu \tau_2)\leq B \}.
$$
\begin{proof} \cite{Browning}
Let $\mu\lambda \neq 0$ and take $(t_0,t_1,t_2,t_3,t_4)\in {\mathbb Z}^5_{\operatorname{prim}}\cap V_y$. Then $t_0=\mu\tau_0$, $t_1=\lambda\tau_0$, $t_2=\lambda \tau_2$, $t_3=\mu \tau_3$, $t_4=-\mu\tau_2-\lambda \tau_3$. \end{proof}

Let $\tilde{B}=B/\sqrt{{\lambda}^2+{\mu}^2}$, $c=\lambda\mu / ({\lambda}^2+{\mu}^2)$. Then in the calculation we may assume that 
$$
\lvert c\rvert \leq 1/2,\quad \tau_0^2+(1-\lvert c\rvert )\cdot (\tau_2^2+\tau_3^2)\leq \tilde{B}^2\quad \mbox{ and }\quad {\lambda}^2+{\mu}^2\leq B^2.
$$

If 
$$
N^{*}(V_y,B)=\# \{\; (\tau_0 , \tau_2 , \tau_3)\in {\mathbb Z}^3 \;\mid \; \tau_0\neq 0 ,\; H(\mu \tau_0,\lambda \tau_0, \lambda \tau_2, \mu \tau_3 , \lambda \tau_3 +\mu \tau_2 )\leq B \},
$$
then
$$
N(V\cap V_y,B)=\frac{1}{2}\sum_{k\ll \tilde{B}}\mu(k)\cdot N^{*}\left( V_y,\frac{B}{k}\right) .
$$

By Euler's summation formula,
\begin{multline*}
N^{*}(V_y,B) = \sum_{\substack{\lvert \tau_0\rvert \leq \tilde{B} \\ \tau_0\neq 0 }} \qquad \sum_{(\tau_2+c\tau_3)^2+\tau_3^2\cdot (1-c^2)\leq {\tilde{B}}^2-\tau_0^2} 1 \\
= \sum_{\substack{\lvert \tau_0\rvert \leq \tilde{B} \\ \tau_0\neq 0 }}\qquad \sum_{\lvert \tau_3 \rvert \leq \frac{\sqrt{{\tilde{B}}^2-\tau_0^2}}{\sqrt{1-c^2}}} \left( 2\cdot \sqrt{{\tilde{B}}^2-\tau_0^2-(1-c^2)\tau_3^2} + O(1) \right) \\
= \sum_{\substack{\lvert \tau_0\rvert \leq \tilde{B} \\ \tau_0\neq 0 }} \left( 2\cdot \int\limits_{-\frac{\sqrt{{\tilde{B}}^2-\tau_0^2}}{\sqrt{1-c^2}}}^{\frac{\sqrt{{\tilde{B}}^2-\tau_0^2}}{\sqrt{1-c^2}}} \sqrt{{\tilde{B}}^2-\tau_0^2-(1-c^2)x^2}\cdot \operatorname{d}x +O(\tilde{B}) \right) \\
= \frac{4\pi B^3}{3\cdot \sqrt{1-c^2}\cdot ({\lambda}^2+{\mu}^2)^{3/2}} + O\left( \frac{B^2}{{\lambda}^2+{\mu}^2} \right) .
\end{multline*}

Hence 
$$
N(V\cap V_y,B)=\frac{2\pi B^3}{3\zeta(3)\cdot \sqrt{1-c^2} ({\lambda}^2+{\mu}^2)^{3/2}} + O\left( \frac{B^2}{{\lambda}^2+{\mu}^2} \right) .
$$

After summing over $y\in {\mathbb P}^1(\mathbb Q)$, this gives the result.  

\section{Tamagawa numbers}

In this section we express, following \cite{Browning}, the leading term of the asymptotic formula in Theorem~\ref{main}, part~$2$, via Tamagawa numbers \cite{Peyre}, \cite{Batyrev}.\\

Let $W\subset {\mathbb P}_{\mathbb Q}^4$ be defined by the equation $t_0^2t_2+t_0t_1t_3+t_1^2t_4=0$, $V=W\setminus \{ t_0=t_1=0 \}$ and $\nu\colon X\to W$ be the normalization. Explicitly, we take $X=\{\; \operatorname{rank}\begin{pmatrix} 
t_0 & t_5 & -t_4\\
t_1 & t_2 & (t_3+t_5)
\end{pmatrix} \leq 1 \} \subset {\mathbb P}^5$ and $P=(0:0:0:0:0:1)$. Then $\nu$ is the projection from $P$ into ${\mathbb P}^4=\{ t_5=0 \}$.\\

We use terminology and notation from \cite{Batyrev}. Let $\mathcal L =\mathcal O (1)$ be the chosen ample invertible sheaf on $V$ metrized as in \cite{Browning}. The following Lemma is proven exactly as in \cite{Browning}.

\begin{lemma} (cf. \cite{Browning})
$X$ is the $\mathcal L$-closure of $V$. $V$ is weakly $\mathcal L$-saturated, not $\mathcal L$-primitive and contains no strongly $\mathcal L$-saturated Zariski open dense subvariety. The fibration $V\to {\mathbb P}^1$, which was used to count rational points on $V$, extends to an $\mathcal L$-primitive fibration $X={\mathbb P^1}\times {\mathbb P^2}\to {\mathbb P}^1$, which is the projection onto the first factor. In particular, $V\cap V_y$ is $\mathcal L$-primitive and ${\alpha}_{\mathcal L}(V\cap V_y )={\alpha}_{\mathcal L}(V)=3$. Moreover,
\begin{gather*}
{\beta}_{\mathcal L}(V\cap V_y ) = \operatorname{rank} \operatorname{Pic} (\mathbb P^2)=1,\\
{\gamma}_{\mathcal L}(V\cap V_y ) = \int_{0}^{\infty}\operatorname{e}^{-3y}\cdot \operatorname{d}y=\frac{1}{3},\\
{\delta}_{\mathcal L}(V\cap V_y ) = \# H^1(\operatorname{Gal}(\overline{\mathbb Q}/{\mathbb Q}),\operatorname{Pic}(\mathbb P^2))=1.
\end{gather*}
The Tamagawa number ${\tau}_{\mathcal L}(V\cap V_y )$, defined in \cite{Batyrev}, coincides with the Tamagawa number ${\tau}_H(\mathbb P^2 )$ defined in \cite{Peyre} with respect to the adelic metric on ${\omega}_{\mathbb P^2}^{-1}$ chosen as in \cite{Browning}. The projection $\mathbb P^2 \to \mathbb P^4$ corresponding to the point $y=(\mu : \lambda)\in \mathbb P^1$ is given by $({\tau}_0 : {\tau}_1 : {\tau}_2)\mapsto (\mu {\tau}_0 : \lambda {\tau}_0 : \lambda {\tau}_1 : (-\mu {\tau}_1 - \lambda {\tau}_2) : \mu {\tau}_2)$. 
Following \cite{Peyre}, Lemma~$2.1.2$, one computes
\begin{gather*}
{\omega}_p({\mathbb P}^2({\mathbb Q}_p))=\frac{\# {\mathbb P}^2({\mathbb F}_p)}{p^2}=\frac{p^2+p+1}{p^2}\quad \mbox{for any prime}\;\; p,\\
{\omega}_{\infty}({\mathbb P}^2({\mathbb R}))=\int_{{\mathbb R}^2}\frac{\operatorname{d}x\operatorname{d}y}{(({\lambda}^2 + {\mu}^2)(1+x^2+y^2)+2\mu\lambda \cdot xy)^{3/2}}=\frac{2\pi}{\sqrt{({\lambda}^2+{\mu}^2)({\lambda}^4 + {\lambda}^2{\mu}^2 + {\mu}^4)}}.
\end{gather*}
Hence 
$$
{\tau}_{\mathcal L}(V\cap V_y ) = \frac{2\pi}{\zeta(3)\cdot \sqrt{({\lambda}^2 + {\mu}^2)({\lambda}^4 + {\lambda}^2{\mu}^2 + {\mu}^4)}}.
$$
\end{lemma}

Thus, the leading coefficient of the asymptotic formula in Theorem~\ref{main}, part~$2$, confirms the prediction of Batyrev and Tschinkel \cite{Batyrev} in this case, up to a factor of $1/3$ coming from ${\gamma}_{\mathcal L}(V\cap V_y )$.

\begin{remark}
The discrepancy with the conjectural form of the leading term in \cite{Batyrev} will be resolved if one redefines the constant ${\gamma}_{\mathcal L}(V )$ in general as follows:
$$
{\gamma}_{\mathcal L}(V ):= {\mathcal X}_{{\Lambda}_{\operatorname{eff}}(V,\mathcal L)}(\tilde{\rho}([{\rho}^{*}L])).
$$
See \cite{Batyrev}, Definition~$2.3.16$. Such a modification is justified by the observation that the constant $c_{{\mathcal L}^k}(V)$, if defined as in \cite{Batyrev}, section~$3.4$, grows linearly with $k$. After this modification, $c_{{\mathcal L}^k}(V)$ grows as $k^{1-{\beta}_{\mathcal L}(V)}$, as needed for the compatibility of the conjecture in \cite{Batyrev} with the equality $N(V,{\mathcal L}^k,B)=N(V,\mathcal L,B^{1/k})$.
\end{remark}

\section*{Acknowledgement}

The author is grateful to Beijing International Center for Mathematical Research, the Simons Foundation and Peking University for support, excellent working conditions and encouraging atmosphere.\\

\bibliographystyle{ams-plain}

\bibliography{NonnormalCubicsRatPts}

\end{document}